\documentclass{amsart}
\usepackage{ amssymb,latexsym, amscd,pb-diagram}
\usepackage[all]{xy}
\usepackage{ stmaryrd }
\usepackage{soul}
\usepackage{enumitem}
\usepackage{amsmath}
\usepackage{multicol}
\usepackage{mathrsfs}
\usepackage{placeins}
\usepackage{tikz}
\usetikzlibrary{matrix}

\newcommand{\co}{{\operatorname{co}}}
\newcommand{\Map}{{\operatorname{Map}}}

\newcommand{\Ker}{{\operatorname{Ker}}}
\newcommand{\Img}{{\operatorname{Im}}}
\newcommand{\spann}{{\operatorname{Span}}}
\newcommand{\Alt}{{\operatorname{Alt}}}
\newcommand{\Z}{\mathbb{Z}}

\newcommand{\Ext}{{\operatorname{Ext}}}

\newcommand{\Der}{{\operatorname{Der}}}
\newcommand{\id}{{\operatorname{id}}}

\newcommand{\Det}{{\operatorname{Det}}}
\newcommand{\Hom}{{\operatorname{Hom}}}
\newcommand{\Tot}{{\operatorname{Tot}}}
\newcommand{\Opext}{{\operatorname{Opext}}}

\newcommand{\vect}[1]{\boldsymbol{#1}}

\newtheorem{theorem}{Theorem}[section]
\newtheorem{proposition}[theorem]{Proposition}
\newtheorem{lemma}[theorem]{Lemma}
\newtheorem{corollary}[theorem]{Corollary}

\theoremstyle{definition}
\newtheorem{definition}[theorem]{Definition}
\newtheorem{example}[theorem]{Example}

\theoremstyle{remark}
\newtheorem{remark}[theorem]{Remark}

\numberwithin{equation}{section}

\begin{document}

\title{five-term exact sequence for Kac cohomology}
\author{C\'esar Galindo}
\address{Departamento de Matem\'aticas\\ Universidad de los Andes\\ Carrera 1 N.
18A - 10,  Bogot\'a, Colombia}
\email{cn.galindo1116@uniandes.edu.co}

\author{Yiby Morales}
\address{Departamento de Matem\'aticas\\ Universidad de los Andes\\ Carrera 1 N.
18A - 10,  Bogot\'a, Colombia}
\email{yk.morales964@uniandes.edu.co}

\thanks{C.G. is partially supported by Faculty of Science of  Universidad de los Andes, Convocatoria 2018-2019 para la Financiaci\'on de Programas de Investigaci\'on, programa ''Simetr\'{i}a $T$ (inversi\'on temporal) en
categor\'{i}as de fusi\'{o}n y modulares''.}

\dedicatory{Dedicated to Nicol\'as Andruskiewitsch on the occasion of his 60th birthday}
\begin{abstract}
We use relative group cohomologies to compute the Kac cohomology of matched pairs of finite groups. This cohomology naturally appears in the theory of abelian extensions of finite dimensional Hopf algebras. We prove that Kac cohomology can be computed using relative cohomology and relatively projective resolutions. This allows us to use other resolutions, besides the bar resolution, for computations. We compute, in terms of relative cohomology, the first two pages of a spectral sequence which converges to the Kac cohomology and its associated five-term exact sequence. Through several examples, we show the usefulness of the five-term exact sequence in computing groups of abelian extensions.
\end{abstract}
\maketitle

\section{Introduction}

Extension theory of groups plays a significant role in the construction and classification of finite groups. In the same way, the extension theory of Hopf algebras has led to results on the still wide open problem of construction and classification of finite-dimensional semisimple Hopf algebras, \cite{Kashina,Masuoka-p3,Natale1,Natale2,Natale-pqr}. The set of equivalence classes of extensions of a group $G$ by a $G$-module $M$ is an abelian group with the Baer product of extensions, which is isomorphic to the second cohomology group $H^2(G,M)$.  A generalization of this theory to Hopf algebras is obtained for the so-called \emph{abelian extensions}, that is,  cleft Hopf algebra extensions of a commutative Hopf algebra $K$ by a cocommutative Hopf algebra  $H$, see \cite{Hofstetter,Kac,Masuoka-Calculations,masuoka-notas,Masuoka-TAMS,Singer}. 

In this paper, we deal with  $H = kF$, a group algebra, and $K = k^G$, the dual of such an algebra, where $F$ and $G$ are finite groups. In this case, each abelian extension has an associated matched pair, that is, a larger group $\Sigma$ such that $G$ and $F$ are subgroups of $\Sigma$ satisfying $G \cap F=\{e\}$ and $\Sigma = GF$. The set of equivalence classes of abelian extensions associated to a fixed matched pair, denoted by $\Opext(kF,k^G)$, is an abelian group that can be computed as the second total cohomology of certain double complex whose cohomology is called \emph{Kac cohomology}.


Obtaining a computation for the $\Opext(kF,k^G)$ of a matched pair of groups can be quite difficult. In fact, there are few general computations in the literature \cite{Masuoka-Calculations}. One obstacle for the computation of  $\Opext(kF,k^G)$ comes from the fact that it is defined as the cohomology of a very specific total complex, and the unique ``cocycle free" tool is the so-called \emph{Kac exact sequence}, (see \cite{Masuoka-Calculations} Corollary \ref{Kac exact sequence}). Perhaps one of first the results that provides a  cocycle free description and interpretation of the Kac cohomology is given in \cite[Proposotion 7.1]{Kac-cohomology}, where the authors describe the Kac cohomology as the \emph{singular cohomology} of the mapping cone $BG\sqcup BF \to B\Sigma$.

In this paper, we use two different kinds of relative cohomology groups: the Auslander relative cohomology and Hochschild relative cohomology (see Section \ref{Section:Kac Cohomomlogy}), in order to compute $\Opext(kF,k^G)$. We prove that $\Opext_{\rhd,\lhd}(kF,k^G)$ can be computed using Auslander relative cohomology and that  Auslander relative cohomology of a matched pair can be computed using relatively projective resolutions. This allows us to use other resolutions, besides the bar resolution, for computing $\Opext_{\rhd,\lhd}(kF,k^G)$.  In addition, we compute the first and second page of a spectral sequence which converges to the Kac cohomology. As a consequence, we compute, in terms of Hochschild relative cohomology, the associated five-term exact sequence, whose second term is $\Opext(kF,k^G)$. In the particular case of a semidirect product, the five-term exact sequence is described in terms of ordinary group cohomology. Finally, doing use of the five-term exact sequence and some non-standard resolutions, we compute  $\Opext(kF,k^G)$ for several families of matched pairs.

The organization of the paper is as follows: in Section \ref{Sec:Preliminaries} we discuss
preliminaries on group cohomology and abelian extensions of Hopf algebras. In Section \ref{Section:Kac Cohomomlogy} we recall the definitions of Auslander relative cohomology \cite{auslander1993relative} and Hochschild relative cohomology \cite{hochschild1956relative}. We prove that $\Opext(kF,k^G)$ can be computed using Auslander Relative cohomology and Auslander relative cohomology of matched pairs can be computed using relatively projective resolutions. In Section \ref{Section:Five-term exact sequence} we compute the first and second page of a spectral sequence which converges to the Kac cohomology. We also compute, in terms of Hochschild relative cohomology, the associated five-term exact sequence, whose second term is $\Opext(kF,k^G)$. Finally, in Section \ref{Section:Computations} we compute  $\Opext(kF,k^G)$ for several families of matched pairs.

\smallbreak\subsubsection*{Acknowledgements} We thank Paul Bressler and Bernardo Uribe for kindly answering some questions and for very useful conversations.
\section{Preliminaries}\label{Sec:Preliminaries}
\subsection{Cohomology of groups} 
Let $G$ be a group and let $M$ be a $G$-module. The $n^{th}$-\emph{cohomology group} of $G$ with coefficients in $M$ is defined as
\begin{equation*}
H^n(G,M)=\Ext^n_{\mathbb{Z}G}(\mathbb{Z},M).
\end{equation*}
We will use occasionally the \emph{normalized bar resolution} $(\mathbb{Z}\stackrel{\epsilon}{\leftarrow}P_i,\delta)$ of $\mathbb{Z}$ as a trivial $G$-module. That is \begin{equation*}P_i=\Z G[G]^i:=\Z G[s_1|\cdots |s_i],\end{equation*} where $s_i\in G$, $s_i\neq e$ for all $i$. The differentials are given by 

\begin{equation*}
\begin{aligned}
\delta([s_1|\ldots|s_{p+1}])=&s_1[s_2|\ldots|s_{p+1}]+\sum_{i=1}^{p}(-1)^i[s_1|\ldots|s_is_{i+1}|\ldots|s_{p+1}]\\
&+(-1)^{p+1}[s_1|\ldots|s_p].
\end{aligned}
\end{equation*}

For a finite cyclic group $C_n=\langle g\rangle$ of order $n$, we occasionally  use a periodic resolution of $\mathbb{Z}$, defined as follows:

\begin{equation}\label{CyclicResolution}
\cdots\stackrel{}{\to}\mathbb{Z}C_n\stackrel{g-1}{\to} \mathbb{Z}C_n\stackrel{N}{\to}\mathbb{Z}C_n\stackrel{g-1}{\to}\mathbb{Z}C_n\stackrel{\epsilon}{\to}\mathbb{Z},
\end{equation}
where, $N=\sum_{i=0}^{n-1} g^i$.  From this, we have that 

\begin{equation}\label{CyclicCohomology}
H^m(C_n,M) = \left\{
\begin{array}{c l}
 M^{C_n}/\Img(N), & m =2k\\
 \Ker(N)/\Img(g-1) & m=2k+1.
\end{array}
\right.
\end{equation}

\subsubsection{Resolutions and cohomology for direct products}
Let $G=G_1\times G_2$ be a direct product of groups. Let $(\mathbb{Z}\stackrel{\epsilon}{\leftarrow}P_i)$ and $(\mathbb{Z}\stackrel{\epsilon}{\leftarrow}Q_i)$ be projective resolutions of $\mathbb{Z}$ as $G_1$-module and $G_2$-module respectively. Then, the total complex $\Tot(P_i\otimes Q_i)$ is a $G$-projective resolution of $\mathbb{Z}$. A useful description of the cohomology for direct products is the following. If $M$ is a trivial $G$-module, then it holds (see e.g. \cite{karpilovsky1985projective})
\begin{equation}\label{Karpilovski}
    H^2(G_1\times G_2,M)\cong H^2(G_1,M)\oplus H^2(G_2,M)\oplus P(G_1,G_2;M),
\end{equation}

where $P(G_1,G_2;M)$ is the group of bicharacters from $G_1\times G_2$ to $M$. An isomorphism is given by $\overline{\alpha}\mapsto (\overline{\alpha_1},\overline{\alpha_2},\phi_\alpha)$, where $\alpha_i$ is the restriction of $\alpha$ to $Q_i\times Q_i$ and $$\phi_{\alpha}(x,y)=\Alt(\alpha)=\alpha(x,y)-\alpha(y,x).$$
the inverse isomorphism is defined by  $(\overline{\alpha_1},\overline{\alpha_2},\phi)\mapsto \overline{\alpha}$ with
$$\alpha((x_1,y_1),(x_2,y_2))=\alpha_1(x_1,x_2)\alpha_2(y_1,y_2)\phi(x_1,y_2).$$

\subsubsection{Second cohomology group and skew-symmetric matrices}

Another useful description of the second cohomology group in the case that $V$ is a finite abelian group is provided by using the universal coefficient theorem. Let $k$ be a field and let us consider the group $k^\times$ of units  of $k$ as a trivial $V$-module. There is a short exact sequence

\begin{equation}\label{UnivCoef}
0\to \operatorname{Ext}(V,k^\times)\to H^2(V,k^\times)\overset{\operatorname{Alt}}{\to} \Hom(\wedge^2(V),k^\times)\to 0.\end{equation}
If $V$ has exponent $n$  and $(k^\times)^n=k^\times$, then 
$\operatorname{Ext}(V,k^\times)=0$. Therefore, the map
  \begin{equation}\label{AltBilForms}
  \Alt:H^2(V,k^\times)\to \wedge^2\widehat{V},\end{equation}
where $\widehat{V}=\Hom(V,k^\times)$, defines an isomorphism $H^2(V,k^\times)\cong \wedge^2\widehat{V}$. \\

\subsection{Extensions of Hopf algebras}\label{HopfAlgExt}

Let $k$ be a field. A sequence of \emph{finite dimensional} Hopf algebras and Hopf algebra maps  \begin{equation*}
    (A):k\to K\overset{i}\rightarrow A\overset{\pi}\rightarrow H \to k
\end{equation*} is called an \emph{extension} of $H$ by $K$
if, $i$ is injective, $\pi$ is surjective, and $K= A^{\co H}$ (see \cite{Andrus-Devoto, Kac, majid1990physics}).

Two extensions $(A),(A')$ of $H$ by $K$ are said to be \emph{equivalent} if there is an homomorphism $f:A\to A'$ of Hopf algebras such that the diagram
\[\xymatrix{
&&A \ar[dd]^{f}\ar[rd]^{\pi}&&\\
k\ar[r]&K \ar[ru]^{i}\ar[rd]_{i'}&&H \ar[r]& k\\
&&A'\ar[ru]_{\pi'}&&
}\]
commutes.
\subsection{Matched pairs of groups}

Let us recall (see e.g., \cite{takeuchi1981matched})  that a \emph{matched pair of groups} is a collection 
 $(G,F,\rhd,\lhd)$ where $G,F$ are groups and $\rhd,\lhd$ are permutation actions \[\xymatrix{G & \ar[l]_{\lhd} G\times F \ar[r]^{\rhd}& F}\] such that
\begin{align*}
s\rhd xy=(s\rhd x)((s\lhd x)\rhd y),&&
st\lhd x=(s\lhd(t\rhd x))(t\lhd x),
\end{align*}for all $s,t \in G, x,y\in F$.

Having a group G and F with a matched pair structure is equivalent to having a group $\Sigma$ with an exact factorization;
the actions $\rhd$ and $\lhd$ are determined by the relations \[sx =(s \rhd x)(s \lhd x),\] $x \in F, s \in G$. 

The group $\Sigma$ associated to a matched pair of groups  will be denoted by $F\bowtie G$, and it is $F\times G$ with product given by \[(x,s)(y,t)=(x(s\rhd y),(s\lhd y)t).\]

It is easy to see that the following
conditions are equivalent:
\begin{itemize}
    \item[(i)]  The action $\rhd$ is trivial.
\item[(ii)] The action $\lhd: G \times F \to G$ is by group automorphisms.
\end{itemize}In this case, the associated group $\Sigma=F \ltimes G$ is a semi-direct product of groups.

\subsection{Abelian extensions}
Let $(G,F,\rhd,\lhd)$ be a matched pair of groups and let us consider $2$-cocycles  $\sigma \in Z^2(F,(k^G)^\times)$, $\tau \in Z^2(G,(k^F)^\times)$. On the vector space \[k^G \#_{\sigma,\tau}  k F:= \operatorname{Span}_k \{e_s\# x : s\in G, x\in F\},\]we can define a unital associative algebra and counital coassociate coalgebra structure by
\begin{align*}
(e_s\#x) (e_t\# y)&= \delta_{s\lhd x,t}\sigma(s;x,y)e_s\# xy,\\
\Delta(e_s\# x)&= \sum_{s=ab}\tau(a,b;x)e_a\# (b\rhd x)\otimes e_b\# x.
\end{align*}Here the 2-cocycles $\sigma$ and  $\tau$ are seen as functions
\begin{align*}
    \sigma: G\times F\times F  \to k^\times,&&
    (s,x,y) &\mapsto \sigma(s;x,y),\\
    \tau:G\times G \times F\to k^\times, && (s,t,x) &\mapsto \tau(s,t;x).  
\end{align*}

The map $\Delta$ is an algebra map if and only if the 2-cocycles satisfy the following compatibility condition
\begin{align*}
\sigma(st;x,y)\tau(s,t;xy) =&\sigma(s;t\rhd x,(t\lhd x)\rhd y)\sigma(t;x,y)\times \\&\tau(s,t;x)\tau(s\lhd(t\rhd x),t\lhd x;y),   \notag
\end{align*}
 for all $x,y \in G, s,t \in F$. In case that the 2-cocycles are compatible, the sequence  \begin{equation*}
k\to k^G\stackrel{i}{\rightarrow}k^G\#_{\sigma,\tau}kF\stackrel{\pi}{\to} kF\to k,
\end{equation*}is a Hopf algebra extension,
where $i(e_s)=e_s\# e$ and $\pi(e_s\# x)=x$. This kind of extension are called \emph{abelian extensions} of Hopf algebras.

The set of equivalence classes of abelian extensions associated to a fixed matched pair $ (G,F,\rhd,\lhd)$ is an abelian group with the Baer product of extensions, and will be denoted by $\Opext_{\rhd,\lhd}(kF,k^G)$ (see \cite{masuoka2002hopf} for mor details).



\section{Kac cohomology and relative cohomology}\label{Section:Kac Cohomomlogy}

In this section, we recall the definitions of two different kinds of relative group group cohomology: the Auslander relative cohomology \cite{auslander1993relative} and the Hochschild relative cohomology \cite{hochschild1956relative}. Our aim  is to prove that $\Opext_{\rhd,\lhd}(kF,k^G)$ can be computed using Auslander Relative cohomology which, in the case of matched pairs, can be computed using relatively projective resolutions. This allows us to use other resolutions besides the bar resolution for computing $\Opext_{\rhd,\lhd}(kF,k^G)$.

\subsection{Auslander relative cohomology of groups}\label{Auslander}


Let $\Sigma$ be a group and $X$ a $\Sigma$-set. We will denote by  $\Lambda_X$ the kernel of the augmentation map
\begin{align}\label{aumentation map}
\epsilon_X:\mathbb{Z}[X] \to \mathbb{Z},&&
x\mapsto 1,
\end{align}
where $\mathbb{Z}[X]$ is the  $\Sigma$-module associated to $X$.

\begin{definition}
Given $\Sigma$-module $A$, the \emph{$n^{th}$-cohomology group of $\Sigma$ relative to $X$ with coefficients in $A$} is defined by 
 \[H_\mathscr{A}^k(\Sigma,X;A):=\operatorname{Ext}_{\mathbb{Z}\Sigma}^{k-1}(\Lambda_X,A), \quad k\geq 1.\]
\end{definition}

Let $X$ be a $\Sigma$-set and $\mathcal{R}_X$ a set of representatives of the $\Sigma$-orbits in $X$. Using Shapiro's Lemma, we have that \[\Ext_\Sigma^k(\Z[X],A)=\prod_{x\in \mathcal{R}_X}\Ext^k_{\Sigma}(\Z[\mathcal{O}(x)],A)\cong  \prod_{x\in \mathcal{R}_X} \Ext^k_{\operatorname{St(x)}}(\Z,A),\] where $\operatorname{St(x)}$ denotes the stabilizer of $x\in X$. Hence, \[\Ext^k(\Z[X],A)=\prod_{x\in \mathcal{R}_X} H^k(\operatorname{St}(x),A).\]
If we apply the functor $\Ext_{\Sigma}(-,A)$ to the exact sequence of $\Sigma$-modules\[ 0\to \Lambda_X\to \Z[X]\to \Z\to 0,\] we obtain the well known long exact sequence of relative cohomology


\begin{equation}\label{long-sequence relative}
    \xymatrix@C-1em{\cdots H^k(\Sigma,A)\ar[r] & \prod_{x \in \mathcal{R}_X} H^k(\operatorname{St}(x),A)\ar[r] & H^{k+1}_{\mathscr{A}}(\Sigma,X,A)\ar[r] & H^{k+1}(\Sigma,A) \cdots   
    }
\end{equation}

\subsection{Hochschild Relative cohomology of groups}\label{Poincare}
Relative cohomology of groups was originally defined by Hochschild \cite{hochschild1956relative} and Adamson \cite{adamson1954cohomology}. We follow  the descrition given in \cite{Alperin-book}.\\ 

Let  $U$ be a $G$-module  and $S$ a subgroup of $G$. We say that $U$ is \emph{relatively $S$-projective} if it satisfies the  following equivalent properties (see  \cite[Proposition 1, page 65]{Alperin-book}):
\begin{itemize}
    \item[(i)] If $\psi:U \twoheadrightarrow V$ is a surjective  $G$-homomorphism and $\psi$ splits as $S$-homomorphisms then $\psi$ splits as $G$-homomorphism.
    \item[(ii)] If $\psi:V \twoheadrightarrow W$ is a surjective $G$-homomorphism and $\phi: U\to W$ is a $G$-homomorphism, then there is a $G$-homomorphism $\lambda:U\to V$ with $\psi \lambda=\phi$, provided that  there is an $S$-homomorphisms $\lambda_0:U\to V$ with the same property. 
    \item[(iii)] $U$ is a direct summand of $U\downarrow_S \uparrow_G.$
\end{itemize}Here, $\downarrow_S$ means the restriction to $S$, and $ \uparrow_G$ the induction to $G$. 

A complex 
\begin{equation*}
\mathscr{R}:\cdots\stackrel{}{\to}R_3\stackrel{\delta_3}{\to} R_2\stackrel{\delta_2}{\to}R_1\stackrel{\delta_1}{\to}R_0\stackrel{\epsilon}{\to}M
\end{equation*}
of $G$-modules is called a \emph{relatively $S$-projective resolution} if: 
\begin{enumerate}
\item each $G$-module $R_i$ is relatively $S$-projective,
\item the sequence has a contracting homotopy as $S$-modules.
\end{enumerate}
\begin{remark}
\begin{itemize}
    \item Since the canonical map  $M\downarrow_S \uparrow_G \to M$  splits as an $S$-homomorphism, if $\mathscr{T}$ is a projective $S$-resolution of $M\downarrow_S$, then $\mathscr{T}\uparrow_G$ is a relatively $S$-projective resolution of $M$.
    \item If $S$ is the trivial subgroup of $G$,  the relatively $S$-projective resolutions of  $G$-module  are the same as projective resolutions of $G$-modules.
\end{itemize}\end{remark}

\begin{definition}
Given a relatively $S$-projective resolution $\mathscr{R}$ of $M$, the $n^{th}$-relative $S$-cohomology group of $G$ is defined by
\[H^m(G,S;M)=H^n(\Hom_{G}(\mathscr{R},M),\delta^*).\]
\end{definition}
As expected, this definition does not depend on the chosen relatively $S$-projective resolution of $M$, (see e.g., \cite{hochschild1956relative}).
\begin{example}[Standard complex \cite{snapper1964cohomology}]\label{Example1}  Let  $X$ be a transitive left $G$-set.  Let $C_i=\mathbb{Z}X^{ (i+1) }$ be the free $\mathbb{Z}$-module generated by all $i+1$-tuples of elements of $X$. The group $G$ acts diagonally on $C_i$ and the sequence
\begin{equation*}
C^X_*:=\cdots\stackrel{}{\to}C_3\stackrel{\delta_3}{\to} C_2\stackrel{\delta_2}{\to}C_1\stackrel{\delta_1}{\to}C_0\stackrel{\epsilon}{\to}\mathbb{Z},
\end{equation*}
where 
\begin{equation}
    \delta_i(x_1,\ldots,x_{r+1})=\sum_{j=1}^{r+1}(-1)^{j+1}(x_1,\ldots,\widehat{x_j},\ldots,x_{r+1})
\end{equation}
and $\epsilon(x)=1$ for all $x\in X$, is a complex of $G$-modules.  

If $F$ denotes the stabilizer subgroup of $x_0\in X$, the complex $C^X_*$ is relatively $F$-projective resolution of $\mathbb{Z}$ called the \emph{standard complex of} $(G,X)$.
\end{example}


\begin{proposition}\label{prop:resultion relativa matched}
Let $\Sigma=F\bowtie G$ be a matched pair and $Q:=(\mathbb{Z}\stackrel{\epsilon_Q}{\leftarrow}Q_i,\delta_i')$ be the normalized right bar resolution of the trivial $G$-module $\mathbb{Z}$. Then the group $\Sigma$ acts on $Q_i$ by
\begin{align}\label{formula:action}
     [s_n|\ldots|s_2|s_1]\cdot (x,s)&=[s_n\lhd\big ( (s_{n-1}\dots s_1)\rhd x\big )|\ldots|s_2\lhd(s_1\rhd x)|s_1\lhd x]s,
\end{align}
and $Q$ is a relatively $F$-projective resolution of right $\Sigma$-modules . 
\end{proposition}
\begin{proof}
Since $\Sigma$ is a matched pair, the right $\Sigma$-set of cosets $F \backslash \Sigma$ can be identified with the set $G$ and $\Sigma$-action $s\cdot (f,g)=(s\lhd f)g$. This $\Sigma$-set will be denoted by $X$.
Applying the construction of Example \ref{Example1} to $X$, we obtain a relatively $F$-projective resolution $C:=C_i\stackrel{\epsilon}{\to} \Z$, since $F$ is the stabilizer of $e\in G$. The resolution $C$ coincides with the standard $G$-free resolution of $\mathbb{Z}$ as a trivial $G$-module. A $G$-basis of $C_i$ (called bar basis) is given by  \[[s_i|\cdots|s_2|s_1]=(s_i\cdots s_1,\cdots ,s_2s_1,s_1,e).\]  The action of $\Sigma$ in this basis is given by \eqref{formula:action}. The normalized bar resolution is a quotient of the bar resolution, and it is easy to see that this is also  relatively $F$-projective.
\end{proof}

\begin{remark}\label{Remark about projectively resolution}
\begin{itemize}
    \item In Proposition \ref{prop:resultion relativa matched}, we may consider the normalized \emph{left bar resolution} $(\mathbb{Z}\stackrel{\epsilon_P}{\leftarrow}P_i,\delta_i')$ for the trivial $G$-module $\mathbb{Z}$:
with action of $\Sigma$ given by 
\begin{align}\label{formula:action2}
     (x,s)\cdot[x_1|x_2|\cdots|x_n]=&x[s\rhd x_1|(s\lhd x_1)\rhd x_2|(s\lhd x_1x_2)\rhd x_3|\cdots\\
     &\cdots|(s\lhd(x_1\cdots x_{n-1})\rhd x_n)],\notag.
\end{align}
This is a relatively $G$-projective resolution.
\item The formulas \eqref{formula:action} and \eqref{formula:action2} appear in \cite{Masuoka-Calculations}.
\end{itemize}

\end{remark}

\begin{theorem}\label{thm:non-standard res}
Let $\Sigma=F\bowtie G$ be a matched pair. Let $(\mathbb{Z}\stackrel{\epsilon_P}{\leftarrow} P_i,\delta_i)$  and  $(\mathbb{Z}\stackrel{\epsilon_Q}{\leftarrow} Q_i,\delta_i')$  be  a relatively  $F$-projective  and a relatively  $G$-projective $\Sigma$-resolutions of $\Z$, respectively. Then the total complex of the tensor product double complex \begin{equation*}
P_{i}\otimes Q_{j}\mbox{ for }i,j\geq 0.\end{equation*}is a projective $\Sigma$-resolution of $\Z$.
\end{theorem}
\begin{proof}
Since $P_i$ is relatively $F$-projective, $P_i\downarrow_F\uparrow^\Sigma=P_i\oplus P'_i$ as $\Sigma$-modules. Analogously,  $Q_j\downarrow_G\uparrow^\Sigma=Q_i\oplus Q'_i$.

Using the fact that  $(F,G)$ is an exact factorization of $\Sigma$ and the Mackey's tensor product theorem (see e.g. \cite[Theorem 10.18]{Curtis-book}), we have that
\begin{align*}
(P_i\otimes Q_j)\downarrow_{\{e\}}\uparrow^\Sigma & \cong P_i\downarrow_F\uparrow^\Sigma\otimes Q_j\downarrow_G\uparrow^\Sigma,\\
&= (P_i\oplus P_i')\otimes (Q_j\oplus Q_j')\\
&= P_i\otimes Q_j \oplus \big (  P_i\otimes Q'_j \oplus P'_i\otimes Q_j \oplus  P_i'\otimes Q_j\oplus P_i'\otimes Q_j'\big ).
\end{align*}Hence $P_i\otimes Q_j$ is direct summand of the $\Sigma$-free module $(P_i\otimes Q_j)\downarrow_{\{e\}}\uparrow^\Sigma$, that is, $P_i\otimes Q_j$ is a projective $\Sigma$-module. 

Finally, since each  $P_i$ and $Q_j$ is a free $\Z$-module,  all rows (and columns) of $P_*\otimes Q_*$ are exact, then it follows from the Acyclic Assembly Lemma (see e.g., \cite[Lemma 2.7.3]{Weibel}) that $\operatorname{Tot}(P_*\otimes Q_*)$ is acyclic.
\end{proof}
Theorem \ref{thm:non-standard res} generalizes the results of \cite{Masuoka2,MR1957830} about the construction of non-standard free resolutions  of $\Z$ associated to a matched pairs of groups. In fact, taking the relatively projective resolutions of Proposition \ref{prop:resultion relativa matched} and Remark \ref{Remark about projectively resolution} we can obtain the resolutions in \cite{Masuoka2,MR1957830}.
\subsection{Kac cohomology as relative group cohomology}

Using the bijective maps,
\begin{align*}
    \Sigma/G\to F,&&  (f,g)G\mapsto f, \\
    F\backslash\Sigma\to G, && F(f,g) \mapsto g
\end{align*}
 we can endow the set $F$ with a left $\Sigma$-action $(f,g)x=f(g\rhd x)$ and $G$ with a right $\Sigma$-action $s(f,g)=(s\lhd f)g$. From now on, $X$ will denote the left $\Sigma$-set defined as the disjoint union $F\sqcup G$, where $G$ is considered a left $\Sigma$-set using the inverse.
We denote by $\Lambda_X$ the kernel of the augmentation map \eqref{aumentation map}.
\begin{proposition}\label{ResDelta}
Let $\Sigma=F\bowtie G$ be a matched pair. Let $(\mathbb{Z}\stackrel{\epsilon_P}{\leftarrow} P_i,\delta_i)$  and  $(\mathbb{Z}\stackrel{\epsilon_Q}{\leftarrow} Q_i,\delta_i')$  be  a relatively  $F$-projective  and a relatively  $G$-projective $\Sigma$-resolutions of $\Z$, respectively.  Let $D^{*,*}$ be the truncated tensor product double complex  
\begin{equation}\label{TruncateComplex}
D^{i,j}:=P_{i+1}\otimes Q_{j+1}\mbox{ for }i,j\geq 0.\end{equation} 
Then, the total complex $(\Tot(D^{*,*}),d_i)$ completed with the map
\begin{equation*}
P_0\otimes Q_0\stackrel{-\delta_1\otimes\delta'_1}{\longleftarrow}\Tot(D^{*,*}),    
\end{equation*}
is a projective $\Sigma$-resolution of $\Lambda_X$.
\end{proposition}

\begin{proof}
Let $\Lambda$ be the kernel of the map $\epsilon:P_0\oplus Q_0\to \mathbb{Z}$ defined by $\epsilon(x\oplus y)=\epsilon_P(x)+\epsilon_Q(y)$.
Let us consider the total complex $(\Tot(D^{*,*}),d_i)$ completed with the maps

\begin{equation}\label{eq:res_delta_proof}
0\stackrel{}{\leftarrow}\Lambda \stackrel{\theta}{\leftarrow}P_0\otimes Q_0\stackrel{-\delta_1\otimes\delta'_1}{\longleftarrow}\Tot(D^{*,*}),  
\end{equation}where $\theta(p\otimes q)=-p\epsilon_Q(q)\oplus \epsilon_P(p)q$.
Using the same argument of Theorem \ref{thm:non-standard res},  we have that $\Tot(D^{*,*})$ is an acyclic complex of  projective $\Sigma$-modules.

Let us see that \eqref{eq:res_delta_proof} is exact in  $P_1\otimes Q_1$. The map $d_1:P_2\otimes Q_1\oplus P_1\otimes Q_2\to P_1\otimes Q_1$ in $\Tot(D^{*,*})$ is defined by $d_1(a\oplus b)=(\delta_2\otimes \id)(a)-(\id\otimes \delta_2')(b)$, where $a\in P_2\otimes Q_1$ and $b\in P_1\otimes Q_2$. Hence, the composed map is given by 
\[(-\delta_1\otimes\delta_1')\circ d_1(a,b)=(-\delta_1\otimes\delta'_1)(\delta_2\otimes \id)(a)+(\delta_1\otimes\delta'_1)(\id\otimes \delta_2')(b)=0.\]
Thus, $\Img(d_1)\subseteq \Ker(-\delta_1\otimes\delta'_1)$. Now, since $P_i$ and $Q_i$ are free $\Z$-modules, then 
\begin{align*}
    \Ker(-\delta_1\otimes\delta_1')&=\Ker(\delta_1)\otimes Q_1+P_1\otimes \Ker(-\delta_1')\\
    &=\Img(\delta_2)\otimes Q_1-P_1\otimes \Img(\delta_2').
\end{align*}
so $\Ker(-\delta_1\otimes\delta'_1)\subseteq \Img(d_1)$. To see the exactness in $P_0\otimes Q_0$, note that 
\[\theta\circ(\delta_1\otimes\delta_1')(p\otimes q)=\delta_1(p)\epsilon_Q(\delta_1'(q))\oplus\epsilon_P(\delta_1(p))\delta_1(q)=0.\]
Hence, $\Img(\delta_1\otimes\delta_1')\subseteq \Ker(\theta)$.
To see that $\Ker(\theta)\subseteq\Img(\delta_1\otimes\delta_1')$, let us consider the tensor product of the two chain complexes $0\leftarrow\mathbb{Z}\stackrel{\epsilon_P}{\leftarrow} P_i$ and $0\leftarrow\mathbb{Z}\stackrel{\epsilon_Q}{\leftarrow} Q_i$. That is,

\begin{equation}\label{Diagram3}
\xymatrix{
  &\mathbb{Z}\otimes Q_1 \ar[d]_{-id\otimes\delta'_1} &P_0\otimes Q_1\ar[d]^{id\otimes\delta'_1} \ar[l]_{\epsilon_P\otimes id} &P_1\otimes Q_1\ar[l]_{\delta_1\otimes id}\ar[d]^{-id\otimes\delta'_1}  \\
    &\mathbb{Z}\otimes Q_0\ar[d]_{-id\otimes\epsilon_Q} &P_0\otimes Q_0\ar[l]^{\epsilon_P\otimes id}\ar[d]^{id\otimes\epsilon_Q}& P_1\otimes Q_0\ar[l]^{\delta_1\otimes id}\ar[d]^{-id\otimes\epsilon_Q}\\
 &\mathbb{Z} &P_0\otimes\mathbb{Z} \ar[l]^{\epsilon_P\otimes id} &P_1\otimes\mathbb{Z}\ar[l]^{\delta_1\otimes id}.\\
}
\end{equation}
whose total complex is given by the exact sequence

\begin{equation}\label{sequece3}0\leftarrow\mathbb{Z}\otimes\mathbb{Z}\stackrel{\epsilon}{\longleftarrow}P_0\otimes\mathbb{Z}\oplus \mathbb{Z}\otimes Q_0\stackrel{d_1}{\longleftarrow}P_1\otimes\mathbb{Z}\oplus P_0\otimes Q_0\oplus \mathbb{Z}\otimes Q_1\leftarrow\cdots\end{equation}

which can be written as
\begin{equation}\label{sequence4}
\mathbb{Z}
\stackrel{\epsilon}{\longleftarrow}P_0\oplus Q_0\stackrel{d_1}{\longleftarrow}P_1\oplus P_0\otimes Q_0\oplus  Q_1\leftarrow\cdots    
\end{equation}
Note that $\theta$ is the restriction of $d_1$ to $P_0\otimes Q_0$. Suppose that $c\in \Ker(\theta)$. Then, $(\id\otimes\epsilon_Q)(c)=(\epsilon_P\otimes\id)(c)=0$. Let $b_2,b_3$ be the preimages of $c$ under the vertical and horizontal differentials in \eqref{Diagram3} respectively. There is a tuple $b=(b_1,b_2,-b_3,-b_4)$ such that $d_2(b)=0$. Since the total complex of \eqref{Diagram3} is acyclic, there is a tuple $a=(a_1,a_2,a_3,a_4,a_5)$ such that $d_3(a)=b$, and, it can be verified that $a_3\in P_1\otimes Q_1$ satisfies $\theta(a_3)=c$. Therefore, $\Ker(\theta)\subseteq\Img(\delta_1\otimes\delta_1')$.

To see that $\theta$ is surjective, note that $\epsilon=\epsilon_P+\epsilon_Q$ in \eqref{sequence4},  and \[d_1(a\oplus b\oplus c)=(\delta_1(a)+(id\otimes \epsilon_Q)(b))\oplus ((\epsilon_P\otimes id)(b) - \delta'_1(c)).\] Let $p_0\in d_1(P_1)=\Ker(\epsilon_P)$ and let $q\in Q_0$ such that $\epsilon_Q(q)=1$. Then, $$d_1(p_0\otimes q)=(-id\otimes\epsilon_Q+\epsilon_P\otimes id)(p_0\otimes q)=p_0,$$ so $d_1(P_1)\subseteq d_1(P_0\otimes Q_0)$. Similarly, $d_1(Q_1)\subseteq d_1(P_0\otimes Q_0)$. Hence, \[\Lambda=\Ker(\epsilon)=\Img(d_1)=\spann\{d_1(P_1)\cup d_1(P_0\otimes Q_0)\cup d_1(Q_1)\}=d_1(P_0\otimes Q_0).\]
Also, the map $d_1$ restricted to $P_0\otimes Q_0$ is given by $-\id\otimes\epsilon_Q\oplus\epsilon_P\otimes \id=\theta$, then $\Lambda=\Img(\theta)$ and therefore the sequence \eqref{eq:res_delta_proof} is exact.

Finally, we see that $\Lambda$ and $ \Lambda_X$ are isomorhpic as $\Sigma$-modules. Let us take $P_*'=C^G_*$ and $Q_*'=C^F_*$, the standard resolutions  as in Example \ref{Example1}, were $F$ and $G$ are consider as $\Sigma$-sets. In this case, \[C_0^G\oplus C_0^F= \Z[G]\oplus \Z[F]\cong \Z[G\sqcup F]=\Z[X],\]
and $\epsilon$ is the augmentation map \eqref{aumentation map}. Then, for this resolution $\Lambda=\Lambda_X$.

If  $(\mathbb{Z}\stackrel{\epsilon_P}{\leftarrow} P_i,\delta_i)$ and $(\mathbb{Z}\stackrel{\epsilon_Q}{\leftarrow} Q_i,\delta_i)$ are relatively projective resolutions, there exists homotopy equivalences $$s: P_*\to C^G_*, \quad l:Q_*\to C^F_*$$ This implies that $P_0\otimes Q_0\stackrel{-\delta_1\otimes\delta'_1}{\longleftarrow}\Tot(D^{*,*})$ is homotopically equivalent to 
$P_0'\otimes Q_0'\stackrel{-\delta_1\otimes\delta'_1}{\longleftarrow}\Tot(D^{*,*})$. 
Hence,
\begin{align*}
\Lambda_X \cong \operatorname{Coker}(P_1'\otimes Q_1'\to P_0'\otimes Q_0')\cong \operatorname{Coker}(P_1\otimes Q_1\to P_0\otimes Q_0)\cong \Lambda.
\end{align*}
\end{proof}

\begin{theorem}\label{OpextRelativeCohomology}
Let $k$ be a field and  $\Sigma= F\bowtie G$ a matched pair of groups. Then,
\[\Opext_{\rhd,\lhd}(kF,k^G)\cong H^3_\mathscr{A}(\Sigma,X;k^\times),\]
where $k^\times$ is considered as a trivial $\Sigma$-module.
\end{theorem}

\begin{proof}
Let $(\mathbb{Z}\stackrel{\epsilon_P}{\leftarrow} P_i,\delta_i)$ and $(\mathbb{Z}\stackrel{\epsilon_Q}{\leftarrow} Q_i,\delta_i´')$ be the resolutions in Proposition \ref{prop:resultion relativa matched}. These are the resolutions used in \cite{Masuoka2} to compute $\Opext_{\rhd,\lhd}(kF,k^G)$: they consider the truncated tensor product $D^{i,j}$ of the two resolutions to get \[\Opext_{\rhd,\lhd}(kF,k^G)\cong H^1(\Tot(\Hom_\Sigma(D^{i,j}))).\]
If $\Lambda$ is the kernel of the map $\epsilon:P_0\oplus Q_0\to \mathbb{Z}$ defined by $\epsilon(x\oplus y)=\epsilon_P(x)+\epsilon_Q(y)$, (here, $P_0:=\mathbb{Z}F=\mathbb{Z}X$ and $Q_0:=\mathbb{Z}Y$ for the $\Sigma$-sets $X=F$ and $Y=G$) then, by Theorem \ref{ResDelta}, the total complex of $E_{i,j}:=P_{i+1}\otimes Q_{j+1}\mbox{ for }i,j\geq 0$ completed in the following way
\begin{equation*}
0{\longleftarrow}\Lambda\stackrel{\theta}{\longleftarrow}P_0\otimes Q_0\stackrel{-\delta_1\otimes\delta'_1}{\longleftarrow}\Tot(D_{*,*}),    
\end{equation*}
is a resolution of $\Lambda$, and $\Lambda=\Lambda_X$. Therefore, if we apply $\Hom_\Sigma(-,k^\times)$ to this total complex, we get the relative cohomology groups $H_\mathscr{A}^n(\Sigma,X;A)$. That is 

\begin{equation}\label{RelativeTot}
    H^k(\Hom_\Sigma(\Tot(D^{*,*})))\cong H_\mathscr{A}^{k+2}(\Sigma,X;A).
\end{equation}
 In particular, since $\Hom_\Sigma(\Tot(D^{*,*}))\cong \Tot(\Hom_\Sigma(D^{*,*}))$, then,
\begin{equation}\label{OpextRelativeCohomologyEq}\Opext_{\rhd,\lhd}(kF,k^G)\cong H^1(\Tot(\Hom_\Sigma(D^{*,*})))\cong H^3_\mathscr{A}(\Sigma,X;k^\times).\end{equation}
\end{proof}

As a consequence of Theorem \ref{OpextRelativeCohomology} and the long exact sequence \eqref{long-sequence relative} we obtain the Kac's exact sequence (see \cite{Masuoka-Calculations,Kac}).
\begin{corollary}[Kac's exact sequence]\label{Kac exact sequence}

For a fixed matched pair of groups $(G,F,\rhd,\lhd)$, we have a long exact sequence
\end{corollary}
\begin{align*}
       0&\to H^1(F\bowtie G,k^\times)\to H^1(F,k^\times)\oplus H^1(G,k^\times)\to H^3_\mathscr{A}(\Sigma,X;k^\times)\\
       &\to H^2(F\bowtie G,k^\times)\to H^2(F,k^\times)\oplus H^2(G,k^\times)\to\Opext_{\rhd,\lhd}(kF,k^G)\\
       &\to H^3(F\bowtie G,k^\times)\to H^3(F,k^\times)\oplus H^3(G,k^\times)\to H^4_\mathscr{A}(\Sigma,X;k^\times).
   \end{align*}

\section{The five-term exact sequence for Kac double complex}\label{Section:Five-term exact sequence}

The group $\Opext_{\rhd,\lhd}(kF,k^G)$ can be obtained, as described in \cite{Masuoka-Calculations}, as the first cohomology group of a double complex, which can be computed by means of a spectral sequence. We compute the first pages of the spectral sequence associated to the double complex $D^{**}$ in $\eqref{TruncateComplex}$, which is a particular case of the double complex of Kac. The five-term exact sequence for this spectral sequence will be useful for computing the group $\Opext_{\rhd,\lhd}(kF,k^G)$ for different kinds of matched pairs. 
\subsection{Spectral sequence of a double complex}
Through this section we deal with a first quadrant double complex, that is, a double complex $M^{p,q}$ such that $M^{p,q}=\{0\}$ when $p,q< 0$. There is a spectral sequence associated to a first quadrant double complex, whose first pages are obtained taking vertical and horizontal cohomology of the double complex.\\

Let $M^{*,*}$ be a first quadrant double complex with vertical and horizontal differentials given by $\delta_v,\delta_h$. Let $\Tot(M^{*,*})$ be the total complex associated to $M^{*,*}$.  
 There is a spectral sequence $(E_r^{*,*},d_r)$ with differentials $d_r^{p,q}:E_r^{pq}\to E_r^{p+r,q-r+1}$, which converges to $H^*(\Tot(M^{*,*}))$, whose first pages are given by
\begin{align*}
E_0^{p,q}=M^{p,q},&&
E_1^{p,q}=H^q(M^{p,*},d_0),&& 
E_2^{p,q}=H^p(E_1^{*,q},d_1),
\end{align*}see \cite{McCleary} for more details. 
The differentials for each page are given by
\begin{align}\label{Differentials}
d_0^{p,q}=d_v,&&
d_1^{p,q}=d_h',&&
d_2^{p,q}(\bar{\alpha})=\overline{d_h(\gamma)},
\end{align}
where $d_h'$ is the differential induced by $d_h$ on $H^q(M^{p,*},d_0)$ and $\gamma\in M^{p+1,q-1}$ is such that $d_h(\alpha)=d_v(\gamma)$.
 Associated to the spectral sequence $(E_r^{*,*},d_r)$, there is a five-term exact sequence
\begin{equation}\label{FTES0}
0{\rightarrow}E_2^{1,0}\stackrel{i}{\to} H^1(\Tot(M^{*,*}))\stackrel{p}{\to}E_2^{0,1}\stackrel{d_2^{0,1}}{\to}E_2^{2,0}\stackrel{i}{\to}H^2(\Tot(M^{*,*})),
\end{equation}
where $i$ is a restriction map, the map $p$ is a projection map.

\subsection{The five-term exact sequence}
Given a matched pair $(G,F,\rhd,\lhd)$ we compute a five-term exact sequence to calculate $\Opext_{\rhd,\lhd}(kF,k^G)$. 


\begin{theorem}\label{FirstSecondPage}
Let $\Sigma=F\bowtie G$ be a matched pair and $A$ be a $\Sigma$-module. The first and second pages of the spectral sequence associated to the double complex $\Hom_\Sigma(D^{*,*},A)$, where $D^{i,j}:=P_{i+1}\otimes Q_{j+1}\mbox{ for }i,j\geq 0,$ is the double complex defined in Proposition \ref{ResDelta}, are given by
\begin{align*}
    E_1^{i,n}=H^n(\Sigma,G;\Hom(P_i,A)), && E_2^{n,m}=H^m(H^n(\Sigma,G;\Hom(P_*,A))),
\end{align*}
The first page does not depend on the resolution $Q_*$ and the second page does not depend neither on the resolution $P_*$ nor on the resolution $Q_*$.
\end{theorem}

\begin{proof}
The double complex $\Hom_\Sigma(D^{*,*},A)$ with only the vertical differentials is the $0^{th}$-page of the spectral sequence:
\begin{equation}\label{HomSigma}E^{i,j}_0:=\Hom_{\Sigma}(P_i\otimes Q_j,A)\cong \Hom_{\Sigma}(Q_j,\Hom_\mathbb{Z}(P_i,A)).\end{equation}
Since $(\mathbb{Z}\stackrel{\epsilon_Q}{\leftarrow} Q_i,\delta_i')$ is a relatively projective resolution of $\mathbb{Z}$, the first page of the spectral sequence is

\[E_1^{i,n}=H^n(\Sigma, G, \Hom(P_i,A))=F_n(P_i).\]
where $F_n:\Sigma$-$Mod$ to $Ab$ is the functor given by $H_1^n(\Sigma,G;\Hom(-,A))$. Hence it does not depend on the resolution $Q_i$. The diagram for the first page is given below.
\begin{equation*}
\xymatrix{
&\cdots  &\cdots &\cdots  \\
&F_2(P_0) \ar[r]^{F_2(\delta_1)}   &F_2(P_1) 
\ar[r]^{F_2(\delta_2)}&F_2(P_2) \ar[r]^{F_2(\delta_3)}& \cdots\\
&F_1(P_0) \ar[r]^{F_1(\delta_1)}  &F_1(P_1) \ar[r]^{F_1(\delta_2)} &F_1(P_2) \ar[r]^{F_1(\delta_3)} &\cdots\\
&F_0(P_0)\ar[r]^{F_0(\delta_1)} &F_0(P_1) \ar[r]^{F_0(\delta_2)} &F_0(P_1) \ar[r]^{F_0(\delta_3)} &\cdots}
\end{equation*}
where $F_k(\delta_i)$ is the induced differential in relative cohomology. Then, the second page is \[E_2^{n,m}=H^m(F_n(P_*)).\] 
To see that the objects in the second page do not depend on the resolution, let $P_i'$ be another relatively $F$-projective resolution of $\mathbb{Z}$. Then, there exists a homotopy equivalence $f:P_i\to P'_i$ as $F$-modules, that is, there exists $g:P'_i\to P_i$ and $h_i:P_i\to P_{i-1}$ such that  
\[\delta_i h+h\delta_i=fg-\id.\]
Since the functor $F_n$ is additive, we get
\[F_n(\delta_i)F_n(h_i)+F_n(h_i)F_n(\delta_i)=F_n(f)F_n(f^{-1})-F_n(\id)\]
for each $n$, so the map $F_n(f)$ is a homotopy equivalence between the resolutions $F_n(P_i)$ and $F_n(P_i')$. This means that the second page, which consist on the cohomology groups of the resolutions $F_n(P_i)$ and $F_n(P'_i)$ with the respective induced differentials, is isomorphic to the first one.
\end{proof}

\begin{theorem}\label{FTESKac}
Let $(G,F,\rhd,\lhd)$ be a matched pair where the action $\rhd$ is trivial and let $k$ be a field. The spectral sequence in Theorem \ref{FirstSecondPage} associated to the group $\Sigma=F\bowtie G$ has second page given by 

\begin{align*}
E_2^{p,q}=&H^{p+1}(F,H^{q+1}(G,k^\times)),&&
E_2^{p,0}\cong H^{p+1}(F,\widehat{G}) \label{SecondPage}\\
E_2^{0,q}\cong& \Der(F,H^{q+1}(G,k^\times)),&&
E_2^{0,0}=\Der(F,\widehat{G})\notag
\end{align*}
for $p\geq 1, q\geq 1$.
Therefore, we have the five-term exact sequence:
\begin{equation}\label{FTES}
\begin{aligned}
0{\rightarrow}H^2(F,\widehat{G}))\stackrel{i}{\to} & \Opext_{\rhd,\lhd}(kF,k^G)\stackrel{\pi}{\to}\Der(F,H^2(G,k^\times))\\
&\stackrel{d_2}{\to}H^3(F,\widehat{G}){\to}H^4_\mathscr{A}(\Sigma,X,k^\times).
\end{aligned}
\end{equation}
\end{theorem}
\begin{proof}
According to \eqref{HomSigma}, the $0^{th}$-page is given by

\begin{equation*}E^{i,j}_0:=\Hom_{\Sigma}(P_i\otimes Q_j,k^\times)\cong \Hom_{\Sigma}(Q_j,\Hom_\mathbb{Z}(P_i,k^\times)).\end{equation*}
If we take $(\mathbb{Z}\stackrel{\epsilon_P}{\leftarrow} P_i,\delta_i)$ and $(\mathbb{Z}\stackrel{\epsilon_Q}{\leftarrow} Q_i,\delta_i´')$ to be the resolutions in Proposition \ref{prop:resultion relativa matched}, then we have the group isomorphism
\begin{equation*}\label{HomMap}
\Hom_{\Sigma}(P_i\otimes Q_j,k^\times)\cong \Map_+(G^{q+1}\times F^{p+1},k^\times)\end{equation*}
and the vertical and horizontal differentials of the double complex of groups 
$$\Map_+(G^{q+1}\times F^{p+1},k^\times)$$ are respectively given by

\begin{equation*}
\begin{aligned}\label{VDifferentials}
\delta_i f(s_{i+1},\cdots,s_1;x_1,\cdots,x_{p})^{(-1)^p}=&f(s_{i+1},\cdots,s_2;s_1\rhd x_1,(s_1\lhd x_1)\rhd x_2,\cdots\\
& \cdots ,(s_1\lhd x_1\cdots x_{p-1})\rhd x_p)\times \\
&\prod_{k=1}^{i}f(s_{i+1},\cdots,s_{i+1}s_i,\cdots,s_1;x_1,\cdots,x_p)^{(-1)^k}\times \\
&f(s_i,\cdots,s_1;x_1,\cdots,x_p)^{(-1)^{q+1}}.
\end{aligned}
\end{equation*}
 \begin{equation*}
\begin{aligned}\label{HDifferentials}
\delta_i' f(s_q,\cdots,s_1;x_1,\cdots,x_{i+1})=&f(s_q\lhd(s_{q-1}\cdots s_1\rhd x_1),\cdots\\
&\cdots ,s_2\lhd(s_1\rhd x_1),s_1\lhd x_1;x_2,\cdots,x_{p+1})\times\\
&\prod_{k=1}^{i}f(s_q,\cdots,s_1;x_1,\cdots,x_ix_{i+1},\cdots,x_{i+1})^{(-1)^k}\times \\
&\  \  \  \  \ \ f(s_q,\cdots,s_1;x_1,\cdots,x_i)^{(-1)^{i+1}}.
\end{aligned}
\end{equation*}
Since the action $\lhd$ is trivial, the vertical  differentials are given by

\begin{equation*}\label{VDifferentials2}
\begin{aligned}
\delta (f)(s_{q+1},\cdots,s_1;x_1,\cdots,x_{p})^{(-1)^p}=&f(s_{q+1},\cdots,s_2;x_1,\cdots,x_p)\\&\prod_{i=1}^{q}f(s_{q+1},\cdots,s_{i+1}s_i,\cdots,s_1;x_1,\cdots,x_p)^{(-1)^i}\\
&f(s_q,\cdots,s_1;x_1,\cdots,x_p)^{(-1)^{q+1}}.
\end{aligned}
\end{equation*}
We have that $\Map_+(G^{q}\times F^{p},k^\times)\cong C^q(G,C^p(F,k^\times))$, where $C^q(G,C^p(F,k^\times))$ denotes the group of functions $f:G^q\to C^p(F,k^\times)$ (with a normalization property) and the group $G$ acts trivially on the group $C^p(F,k^\times)$.

Taking vertical cohomology to the $0$th-page, we get $H^q(G,C^p(F,k^\times))$. Therefore, the first page $E_1^{*,*}$ of the spectral sequence is given by
\begin{align*}E_1^{p,q}&=H^{q+1}(G,C^{p+1}(F,k^\times))\cong C^{p+1}(F,H^{q+1}(G,k^\times))\mbox{   for }q\geq 1.\\
E_1^{p,0}&=\Der(G,C^{p+1}(F,k^\times))\cong C^{p+1}(F,\Der(G,k^\times)),\notag\end{align*}
where the isomorphisms hold since the vertical differential leaves every element of $F$ fixed. On the other hand, the horizontal differentials are given by

\begin{equation*}\label{HDifferentials2}
\begin{aligned}
\delta'(f)(s_q,\cdots,s_1;x_1,\cdots,x_{p+1})=&f(s_q\lhd x_1,\cdots,s_1\lhd x_1;x_2,\cdots,x_{p+1})\\
&\prod_{i=1}^{p}f(s_q,\cdots,s_1;x_1,\cdots,x_ix_{i+1},\cdots,x_{p+1})^{(-1)^i}\\
&f(s_q,\cdots,s_1;x_1,\cdots,x_p)^{(-1)^{p+1}},
\end{aligned}
\end{equation*}
by differentiating each row by the induced horizontal differentials,

\begin{align}
E_2^{p,q}=&H^{p+1}(F,H^{q+1}(G,k^\times)),&&
E_2^{p,0}\cong H^{p+1}(F,\widehat{G})\\
E_2^{0,q}\cong& \Der(F,H^{q+1}(G,k^\times)),&&
E_2^{0,0}=\Der(F,\widehat{G})\notag
\end{align}
Since $k^\times$ is a trivial $G$-module, the sequence \eqref{FTES0} turns into
\begin{equation*}
\begin{aligned}
0{\rightarrow}H^2(F,\Der(G,k^\times))\stackrel{i}{\to} &H^1(\Tot(\Hom_{\Sigma}(P_i\otimes Q_j,k^\times)))\stackrel{p}{\to}\Der(F,H^2(G,k^\times))\\
&\stackrel{d_2^{0,1}}{\to}H^3(F,\Der(G,k^\times))\stackrel{i}{\to}H^2(\Tot(\Hom_{\Sigma}(P_i\otimes Q_j,k^\times))),
\end{aligned}
\end{equation*}
From \eqref{RelativeTot} and \eqref{OpextRelativeCohomologyEq} we get the five-term exact sequence

\begin{equation*}
\begin{aligned}
0{\rightarrow}H^2(F,\widehat{G})\stackrel{}{\to} & \Opext_{\rhd,\lhd}(kF,k^G){\to}\Der(F,H^2(G,k^\times))\\
&\stackrel{d_2^{0,1}}{\to}H^3(F,\widehat{G}){\to}H^4(\Sigma,X,k^\times).
\end{aligned}
\end{equation*}
\end{proof}
Note that, in the case that $\rhd$ is a trivial action, the terms $E_2^{p,q}$ with $p\geq 1,q\geq 1$ of the second page of the spectral sequence associated to the semidirect product $\Sigma=F\ltimes G$ coincides with the second page of the Lyndon-Hochschild-Serre spectral sequence \cite{Evens}. 
\begin{corollary} Let $(G,F,\rhd,\lhd)$ be a matched pair with trivial $\rhd$ action and let $k$ be a field. Then,
\begin{enumerate}
\item If $H^2(G,k^\times)=1$ then   $\Opext_{\rhd,\lhd}(kF,k^G)\cong  H^2(F,\widehat{G}).$
\item If $(|F|,|\widehat{G}|)=1$, then $\Opext_{\rhd,\lhd}(kF,k^G)\cong\Der(F,H^2(G,k^\times)).$
\item If $G$ is a perfect group, then $    \Opext_{\rhd,\lhd}(kF,k^G)=\Der(F,H^2(G,k^\times))$
\item If $|F|=2k+1$ and $G=S_n$ with $n\geq 4$, then, $\Opext_{\rhd,\lhd}(\mathbb{C}F,\mathbb{C}^G)=\Hom(F,\mathbb{Z}/2)$.
\end{enumerate}
\end{corollary}
\begin{proof}
Part $(1)$ is straightforward.
$(2)$:  Since $(|F|,|\widehat{G}|)=1$, then $H^n(F,\widehat{G})=H^n(F,\widehat{G})=\{1\}$ and the result holds.\\

$(3)$: Since the abelianization of $G$ is trivial, then $\widehat{G}\cong\{0\}$. Therefore, $H^2(F,\widehat{G})=H^2(F,\widehat{G})=\{0\}$, similarly $H^3(F,\widehat{G})=0$, then, $\Opext_{\rhd,\lhd}(kF,k^G)=\Der(F,H^2(G,k^\times))$.\\

$(4)$: Under the given conditions it holds $(|F|,|\widehat{G}|)=1$, as in $(3)$. So $\Opext_{\rhd,\lhd}(kF,k^G)=\Der(F,H^2(G,k^\times))$. Now, $H^2(S_n,k^\times)= \mathbb{Z}/2$, for $n\geq 4$ and $\Der(F,H^2(G,k^\times))=\Hom(F,H^2(G,k^\times))$, so the result holds.
\end{proof}

\section{Computations}\label{Section:Computations}

We compute some examples of the group $\Opext_{\rhd,\lhd}(kF,k^G)$ for different semidirect products: the right action $\rhd$ is trivial, so we denote the group $\Opext_{\rhd\lhd}(kF,k^G)$ by $\Opext_\lhd(kF,k^G)$). The first calculation generalizes one from Masuoka in \cite{Masuoka2}.
\begin{theorem}
Let $k$ be a field. Let $G$ be a  group and  $\mathbb{Z}/2\ltimes(G\times G)$ be the the semidirect product with $(a,b)\lhd 1=(b,a).$
  Then
  $$\Opext_\lhd(k\mathbb{Z}/2,k^{G\times G})\cong H^2(G,k^\times )\oplus P(G,G;k^\times).$$
  \end{theorem}
    \begin{proof}
    
It follows from \eqref{CyclicCohomology} that  $H^n(\mathbb{Z}/2,H^1(G\times G,k^\times))=0$ for $n\geq 1$. Then, the sequence \eqref{FTES} implies that 

\[\Opext_\lhd(k\mathbb{Z}/2,k^{G\times G})\cong \Der(\mathbb{Z}/2,H^2(G\times G,k^\times)).\]

According to \eqref{Karpilovski}, we have $H^2(G\times G, k^\times)= H^2(G,k^\times)\oplus P(G,G,k^\times)\oplus H^2(G,k^\times)$. If $F$ is a projective $G$-resolution of $\Z$, the chain map \begin{align*}
(F\otimes F)_n:= \bigoplus_{i+j=n} F_i\otimes F_j &\to \bigoplus_{i+j=n} F_i\otimes F_j\\
f_i\otimes f_j & \mapsto (-1)^{ij}f_j\otimes f_i, 
\end{align*} defines a compatible action of $\mathbb{Z}/2$ on $F\otimes F$. Then, the action of $\mathbb{Z}/2$ on $H^2(
G\times G, k^\times)= H^2(G,k^\times)\oplus P(G,G;k^\times)\oplus H^2(G,k^\times)$ is given by \begin{align*}
    \tau: H^2(G\times G, k^\times)&\to H^2(G\times G, k^\times)\\
    (a,b,c) &\mapsto (c,-b,a).
\end{align*}
Since $\Der(\mathbb{Z}/2,H^2(G\times G,k^\times))=\{s\in H^2(G\times G, k^\times): a+\tau(a)=0\},$ we have that \[\Der(\mathbb{Z}/2,H^2(G\times G,k^\times))=\{(a,b,-a):a\in  H^2(G,k^\times), b\in P(G,G;k^\times) \}.\]
In conclusion, we have that $$\Opext_\lhd(k\mathbb{Z}/2,k^{G\times G})\cong H^2(G,k^\times )\oplus P(G,G;k^\times).$$
\end{proof}
 The following example includes the previous one in the case that $b=c=1$, $a=0$.
  \begin{theorem}
Let $\mathbb{Z}/2\ltimes (\mathbb{Z}/n\oplus \mathbb{Z}/n)$ be the semidirect product where the action of $\mathbb{Z}/2$ on $(\mathbb{Z}/n\oplus \mathbb{Z}/n)$ is defined by the matrix  
\[ A=\left( \begin{array}{cc}
a & b\\ 
c & -a \end{array} \right)\]  with $\Det(A)=-1$. Let $k$ be a field such that $k^\times/(k^\times)^{2n}=0$ Then, $$\Opext_\rhd(k\mathbb{Z}/2,k^{\mathbb{Z}/n\oplus \mathbb{Z}/n})\cong \frac{\operatorname{Ker}(A-I)}{\operatorname{Im}(A+I)}
      \oplus \mu_n(k),$$
      where $\mu_n(k)$ is the group of $n^{th}$-roots of unity in $k$.
\end{theorem}
\begin{proof}
In this case the sequence \eqref{FTES} is given by
\begin{equation*}
\begin{aligned}
0{\rightarrow}&H^2(\mathbb{Z}/2,\Der(\mathbb{Z}/n\oplus \mathbb{Z}/n,k^\times))\stackrel{i}{\to}  \Opext_\lhd (k\mathbb{Z}/2,k^{\mathbb{Z}/n\oplus \mathbb{Z}/n})\\&\stackrel{\pi}{\to}\Der(\mathbb{Z}/2,H^2(\mathbb{Z}/n\oplus \mathbb{Z}/n,k^\times))\stackrel{d_2^{0,1}}{\to}H^3(\mathbb{Z}/2,\Der(\mathbb{Z}/n\oplus \mathbb{Z}/n,k^\times))\\ &{\to}H^4(\Sigma,X,A),
\end{aligned}
\end{equation*}

We will see that 
\begin{itemize}
    \item[(i)]  $H^2(\mathbb{Z}/2,H^1(\mathbb{Z}/n\oplus \mathbb{Z}/n,k^\times))\cong \frac{\operatorname{Ker}(A-I)}{\operatorname{Im}(A+I)}$.
    \item[(ii)] $\Der(\mathbb{Z}/2,H^2(\mathbb{Z}/n\oplus \mathbb{Z}/n,k^\times))\cong \mu_n(k)$,
    
    \item[(iii)] $d_2^{0,1}=0$,
    
\end{itemize}
Therefore, $\Opext_\lhd (k\mathbb{Z}/2,k^{\mathbb{Z}/n\oplus \mathbb{Z}/n})$ fits in a short exact sequence
\begin{equation}\label{ShortExactSequence}
0{\rightarrow}\frac{\operatorname{Ker}(A-I)}{\operatorname{Im}(A+I)}\stackrel{i}{\to}\Opext_\lhd (k\mathbb{Z}/2,k^{\mathbb{Z}/n\oplus \mathbb{Z}/n})\stackrel{\pi}{\to}\mu_n(k)\to 0,
\end{equation}
moreover, we will see  \eqref{ShortExactSequence} is split, so 
$$\Opext_\lhd (k\mathbb{Z}/2,k^{\mathbb{Z}/n\oplus \mathbb{Z}/n})\cong \frac{\operatorname{Ker}(A-I)}{\operatorname{Im}(A+I)}\oplus \mu_n(k).$$

(i) It follows immediately from\eqref{CyclicCohomology}.

(ii) We identify $\wedge^2(\mathbb{Z}/n\oplus\mathbb{Z}/n)$ with the abelian group of alternating $2\times 2$-matrices over the ring $\mathbb{Z}/n\mathbb{Z}$. Therefore, $\wedge^2(\mathbb{Z}/n\oplus\mathbb{Z}/n)\cong \mathbb{Z}/n$ and  \[H^2(\mathbb{Z}/n\oplus\mathbb{Z}/n,k^\times)\cong \Hom(\wedge^2(\mathbb{Z}/n\oplus\mathbb{Z}/n),k^\times)\cong \mu_n(k),\]where $\mu_n(k)$ is the group of $n$-th roots  unit. Since  $A^TMA=-M$  for all $M\in \wedge^2(\mathbb{Z}/n\oplus\mathbb{Z}/n)$ we have that \[\Der(\mathbb{Z}/2,H^2(\mathbb{Z}/n\oplus \mathbb{Z}/n,k^\times))\cong \mu_n(k).\]

(iii) To compute \[d_2^{0,1}:\mu_n(k)=\Der(\mathbb{Z}/2,H^2(\mathbb{Z}/n\oplus \mathbb{Z}/n,k^\times))\stackrel{}{\to}H^3(\mathbb{Z}/2,\Der(\mathbb{Z}/n\oplus \mathbb{Z}/n,k^\times)),\] we follow \eqref{Differentials}. Given $\zeta\in \mu_n(k)=\Der(\mathbb{Z}/2,H^2(\mathbb{Z}/n\oplus \mathbb{Z}/n,k^\times))$, we need to find $\gamma\in \Map_+((\mathbb{Z}/n\oplus \mathbb{Z}/n)\times (\mathbb{Z}/2)^{2},k^\times)\cong C^2(\mathbb{Z}/2,C^1(\mathbb{Z}/n\oplus \mathbb{Z}/n,k^\times))$ such that  \begin{equation}\label{DeltaGamma}\delta_h(\alpha_\zeta)=\delta_v(\gamma),\end{equation}where 
\begin{align}\label{alpha-zeta}
    \alpha_\zeta\in C^2(\Z/n\oplus \Z/n, k^\times), && \alpha_\zeta(x,y)=\zeta^{x_1 y_2}
\end{align}
and then compute the cohomology class of $\delta_h(\gamma)$ in $H^3(\mathbb{Z}/2,\Der(\mathbb{Z}/n\oplus \mathbb{Z}/n,k^\times))$.
We have that
\begin{align*}
\delta_h(\alpha_\zeta)(1,1)(x,y) &=\alpha_\zeta(Ax,Ay)\alpha_\zeta(x,y)\\&= \zeta^{x^T\begin{pmatrix}ac&bc\\bc&-ab\end{pmatrix}y}.    
\end{align*}
This is a bicharacter with associated quadratic form  \[\omega(x,y)=\zeta^{-acx^2-2bcxy+aby^2}.\] Therefore, the cochain $\gamma\in C^2(\Z/2,C^1(\Z/n\oplus\Z/n,k^\times ))$  defined by
\begin{equation}\label{gamma}
\gamma(1,1)= \zeta^{-\frac{acx^2}{2}-bcxy+\frac{aby^2}{2}}, \quad  \quad (x,y)\in \Z/n\oplus \Z/n,
\end{equation}
and $\gamma(0,1)=\gamma(1,0)=\gamma(0,0)=1,$ satisfies \eqref{DeltaGamma}.

Finally, the horizontal differential of $\gamma$ is given by
\begin{align*}
\label{Dif3Gamma}\delta_h(\gamma)(1,1,1)&=\gamma(1,1)\gamma(1,0)(^1(\gamma(1,1))\gamma(0,1))^{-1}\\
&=\gamma(1,1)(^1(\gamma(1,1))^{-1}=1. \notag
\end{align*}
Hence, $d_2(\zeta)=1$.

A section of $\pi$ in the exact sequence \eqref{ShortExactSequence} is given by cohomology of $s(\zeta)=(\alpha_\zeta,\gamma_\zeta)$, where $\alpha_\zeta$ and $\gamma_\zeta$ are given by \eqref{alpha-zeta} and \eqref{gamma}, respectively.  It is clear from the definition that $s$ is a group homomorphism, that is, \eqref{ShortExactSequence} splits.
\end{proof}
\begin{theorem}\label{thm:odd groups}
Let $F$ be an arbitrary group acting on a finite abelian group $V$ with odd order. Suppose that  $(k^\times)^n=k^\times$, where $n$ is the exponent of $V$. Then
 \begin{align*}
   \Opext_\lhd (k F,k^{V})\cong  & H^2(F, \widehat{V})
      \oplus \Der(F,\wedge^2 \widehat{V}),
  \end{align*}where $\widehat{V}=\Hom(V,k^\times)$.   
  \end{theorem}
  
 \begin{proof}
By \eqref{UnivCoef}, we have $H^2(V,k^\times)\cong \wedge^2\widehat{V}$. The sequence \eqref{FTES} is given by:\\
  \begin{equation*}
\begin{aligned}
0{\rightarrow}H^2(F,\widehat{V})\stackrel{i}{\to} & \Opext_\lhd(kF,k^V)\stackrel{\pi}{\to}\Der(F,\wedge^2\widehat{V})\\
&\stackrel{d_2}{\to}H^3(F,\widehat{V}){\to}H^4(\Sigma,X,A),
\end{aligned}
\end{equation*}
where $\Sigma= V \rtimes F$. We will see that $d_2=0$ and, since the resulting short exact sequence splits, we get the result.
  
Let $\alpha\in \Der(F,\wedge^2 \widehat{V})$, that is, $\alpha:F\to \wedge^2\widehat{V}$ such that $\alpha(gh)=\ ^g\alpha(h)\alpha(g)$. By \eqref{AltBilForms}, $\alpha$ can be identified with a map $\alpha:F\to H^2(V,k^\times)$ which can be lifted to a map $\tilde{\alpha}:F\to Z^2(V,k^\times)$ considering that, since $V$ has odd order, the map
  \begin{align*}
      \Alt: \wedge^2\widehat{V}&\to \wedge^2\widehat{V}
  \end{align*}
  given by $\Alt(\phi)(x,y)=\frac{\phi(x,y)}{\phi(y,x)}=\phi(x,y)^2$ is an isomorphism, so we can define the lifting map by $\tilde{\alpha}:F\to Z^2(V,k^\times)$ by
\begin{equation*}
  \tilde{\alpha}(g)=\alpha(g)^{\frac{1}{2}}.
  \end{equation*}
  In order to compute $d_2(\alpha)$, we must find a function $\gamma\in C^2(F,C^1(V,k^\times))$ such that
  \begin{align*}
    \delta(\gamma(g,h))&=\frac{^g\widetilde{\alpha}(h)\widetilde{\alpha}(g)}{\widetilde{\alpha}(gh)}\\
    &=\frac{\widetilde{^g\alpha}(h)\widetilde{\alpha}(g)}{\widetilde{\alpha}(gh)}\\
    &=\left(\frac{^g\alpha(h)\alpha(g)}{\alpha(gh)}\right)^\frac{1}{2}=1.
      \end{align*}
Hence $\gamma$ can be taken to be the constant cochain and, therefore, $d_2(\alpha)=1$ for all $\alpha\in \Der(F,\wedge^2\widehat{V})$.
\end{proof} 
  
   \begin{corollary}
	Let $F=C_m=\langle \sigma \rangle $ be a cyclic group of order $m$ acting on a finite abelian group $V$ with odd order. Suppose that  $(k^\times)^n=k^\times$, where $n$ is the exponent of $V$. Then

 \begin{align*}
   \Opext_\lhd (k F,k^{V})\cong  & \{\psi \in \widehat{V}: \sigma \psi=\psi\}/\{N_{\sigma} \psi: \psi \in \widehat{V}\} \\ & \oplus \{b \in \wedge^2 \widehat{V}: N_{\sigma}  b=0 \},
  \end{align*}where $N_\sigma= 1+ \sigma +\cdots +\sigma^{m-1}$.   
\end{corollary}

\subsection{An example with non-trivial differential $d_2$}
The next example illustrates the fact that the hypothesis in Theorem \ref{thm:odd groups} stating that the order of the order of $V$ must be odd, can not be avoided since otherwise the differential $d_2$ can be not trivial.
\begin{remark}\label{Rmk:matrices alternantes}
\begin{enumerate}[leftmargin=*,label=\rm{(\alph*)}]
\item Let $G$ be an elementary abelian $p$-group of rank $n$. Once a basis of $G$ is fixed, using the isomorphism \eqref{AltBilForms}  we can identify $H^2(G,\mathbb{C}^\times)$ with alternating matrices  over $\mathbb{Z}/p$ . A representative $2$-cocycle $\alpha_M\in H^2(G,\mathbb{C}^\times)$ corresponding to a matrix $M$ is defined by
\begin{equation}\label{2CocycleToAlternatingMatrix}
    \alpha_M(\textbf{x},\textbf{y})=\exp\Big ( \frac{2\pi i}{n} {\bf{x}^T} \widetilde{M} {\bf{y}}\Big ),
\end{equation}
where $\widetilde{M}$ is the upper triangular part of $M$.

\item Let $F=\langle t_1\rangle \oplus \langle t_2\rangle$ be a product of cyclic groups and let $M$ be an left $F$-module. If  $(\mathbb{Z}\stackrel{\epsilon}{\leftarrow}P_i)$ and $(\mathbb{Z}\stackrel{\epsilon}{\leftarrow}P_i')$ are periodic resolutions as in \eqref{CyclicResolution} for the groups $\langle t_1\rangle$ and $\langle t_2\rangle$ respectively, then the total complex $\Tot(P\otimes P')$  is a free $F$-resolution of $\mathbb{Z}$. Therefore, given a $F$-module $M$, we can compute $H^*(F,M)$ as the   cohomology of total complex of
\begin{equation*}\label{MDiagram}
\xymatrix{
&\vdots  &\vdots  &\vdots  \\
&M \ar[r]^{t_1-1}\ar[u]^{t_2-1}   &M
\ar[u]^{t_2-1}\ar[r]^{N_{t_1}}&M\ar[r]^{t_1-1}\ar[u]^{t_2-1}& \cdots\\
&M \ar[r]^{t_1-1} \ar[u]^{N_{t_2}} &M\ar[u]^{N_{t_2}} \ar[r]^{N_{t_1}} &M\ar[u]^{N_{t_2}} \ar[r]^{t_1-1} &\cdots\\
&M\ar[r]^{t_1-1}\ar[u]^{t_2-1} &M\ar[u]^{t_2-1} \ar[r]^{N_{t_1}} &M\ar[u]^{t_2-1} \ar[r]^{t_1-1} &\cdots}
\end{equation*} 
Since we are mainly interested in $H^2(F,M)$, we  write the second and third differentials $\delta_2:M\oplus M\to M\oplus M\oplus M$, $\delta_3:M\oplus M\oplus M\to M\oplus M\oplus M\oplus M$ of the total complex are given by 
\begin{align}
\delta_2(A,B)&=(A+\ ^{g_1}A,\ A-\ ^{g_2}A-(B-\ ^{g_1}B),B+\ ^{g_2}B),\label{eq: diferencial 2 prod cyclic}\\
\delta_3(A,B,C)&=(\ ^{g_1}A-A,\ ^{g_2}A-A+B+\ ^{g_1}B,B+\ ^{g_2}B+\ ^{g_1}C-C,\ ^{g_2}C-C).\label{eq: diferencial 3 prod cyclic}
\end{align}

\end{enumerate}

\end{remark}

\begin{lemma}\label{TwoGroups}
Let $F=\langle t_1,t_2\rangle$, $G=\langle s_1,\ldots,s_4\rangle$ be elementary abelian $2$-groups of rank 2 and 4, respectively. Consider the (right) action of $F$ on $G$ determined by the matrices

\begin{align*}
F_1=\left( \begin{array}{cccc}
1 & 0 & 0 & 0 \\
0 & 1 & 0 & 0\\
1 & 1 & 1 & 0\\
0 & 1 & 0 & 1\end{array} \right),&& F_2=\left( \begin{array}{cccc}
1 & 0 & 0 & 0 \\
0 & 1 & 0 & 0\\
1 & 0 & 1 & 0\\
1 & 1 & 0 & 1\end{array} \right).
\end{align*}
and the induced left action of $F$ on $H^2(G,\mathbb{C}^\times)$. Then,
\begin{enumerate}
\item The group $\Der(F,H^2(G,\mathbb{C}^\times))$ is in correspondence with the set of pairs of matrices
\begin{align*}
    A=\left( \begin{array}{cccc}
0 & 0 & b & c \\
0 & 0 & d & e\\
b & d & 0 & f\\
c & e & f & 0\end{array} \right),&&  B=\left( \begin{array}{cccc}
0 & 0 & b' & c' \\
0 & 0 & d' & e'\\
b' & d' & 0 & f'\\
c' & e' & f' & 0\end{array} \right),
\end{align*} with entries in $\mathbb{Z}/p$, such that


\begin{align}\label{EquationsExample}
    c'+b'+d'&=0\\
    c+d+e&=0\notag\\
    b+e+c'+d'+e&=0.\notag
\end{align}
\item The group $H^2(F,H^1(G,\mathbb{C}^\times))$ is isomorphic to  $\Z/2^2.$
\end{enumerate}
\end{lemma}
\begin{proof}
(1) \ By \eqref{eq: diferencial 2 prod cyclic}, elements in $\Der(F,H^2(G,\mathbb{C}^\times))$ are in correspondence with pairs $(A, B)$ of alternanting $4\times 4$-matrices such that     

\begin{align}\label{equationsAB}
    A+F_1AF_1^T=0\notag\\ 
    B+F_2BF_2^T=0\\
    A- F_2AF_2^T-B+F_1BF_1^T=0.\notag
\end{align} 
The system \eqref{equationsAB} is equivalent to \eqref{EquationsExample}.

$(2)$ In order to compute $H^2(F,H^1(G,\mathbb{C}^\times))$ we use the canonical identification $H^1(G,\mathbb{C}^\times)=\operatorname{Hom}(G,\mathbb{C}^\times)\cong G$ as left $F$-modules. By \eqref{eq: diferencial 3 prod cyclic}, we have that $\Ker(\delta_3)$ is in correspondence with $4\times 3$-matrices $S=[n_a,n_b,n_c]$ over $\mathbb{Z}/2$ such that  
\begin{align*}\label{eq:sistema lineal Ker}
    (F_1-\operatorname{I})n_a&=0, &&
(F_2-\operatorname{I})n_c=0,
\\
(\operatorname{I}-F_1)n_c&=(F_2+\operatorname{I})n_{b},&&
(\operatorname{I}+F_1)n_{b}=(\operatorname{I}-F_2)n_a.\notag
\end{align*}
Thus, the space $\Ker(\delta_3)$ corresponds with all $4\times 3$-matrices over $\mathbb{F}_2$ such that $S_{ij}=0$ for $1\leq i\leq 2$, $1\leq j \leq 3$. On the other hand, by \eqref{eq: diferencial 2 prod cyclic} we have 
\[\Img(\delta_2)= \{(l_a+F_1l_a,l_a-l_b+F_1l_b-F_2l_a,l_b+F_2l_b): l_a, l_b\in N\},\]
that is, $\Img(\delta_2)$ is in correspondence with all matrices of the form 

\[\left(\begin{array}{ccc}
0 & 0 & 0  \\
0 & 0 & 0 \\
x_1+x_2 &  x_1+y_1+y_2 & y_1 \\
x_2 &  x_1+x_2+y_2 & y_1+y_2  \end{array}\right )\]where $x_i, y_i \in \mathbb{Z}/2$. Hence $H^2(F,H^1(G,\mathbb{C}^\times))\cong\Z/2^2$.
\end{proof}



\begin{lemma}\label{Example2}
 Let $\Sigma=F\ltimes G$ be a semidirect product and let
\begin{equation}\label{Resol}
\cdots\stackrel{}{\to}R_3\stackrel{}{\to} R_2\stackrel{}{\to}R_1\stackrel{}{\to}R_0,
\end{equation}
be a free resolution of a right $F$-modules $M$. The action of $F$ on $R_i$ can be extended to an action of $\Sigma$ by $r\cdot (f,g)=r\cdot f$. With this action, the sequence \eqref{Resol} turns out to be a relatively $G$-projective resolution of the right $\Sigma$-module $M$.
\end{lemma}
\qed

\begin{theorem}Let $F=\langle t_1,t_2\rangle$, $G=\langle s_1,\ldots,s_4\rangle$ be elementary abelian $2$-groups of rank 2 and 4, respectively. Consider the (right) action of $F$ on $G$ determined by the matrices

\begin{align*}
F_1=\left( \begin{array}{cccc}
1 & 0 & 0 & 0 \\
0 & 1 & 0 & 0\\
1 & 1 & 1 & 0\\
0 & 1 & 0 & 1\end{array} \right),&& F_2=\left( \begin{array}{cccc}
1 & 0 & 0 & 0 \\
0 & 1 & 0 & 0\\
1 & 0 & 1 & 0\\
1 & 1 & 0 & 1\end{array} \right).
\end{align*}
Then  $\Opext_\lhd(\mathbb{C}F,\mathbb{C}^G)\cong (\Z/2)^3\oplus (\Z/4)^2$.

\end{theorem}
\begin{proof}
The five-term exact sequence \eqref{FTES} for this case is
\begin{equation*}
\begin{aligned}
0{\rightarrow}H^2(F,\Der(G,\mathbb{C}^\times))\stackrel{i}{\to} & \Opext_\lhd(\mathbb{C}F,\mathbb{C}^G)\stackrel{\pi}{\to}\Der(F,H^2(G,\mathbb{C}^\times))\\
&\stackrel{d_2}{\to}H^3(F,\Der(G,\mathbb{C}^\times)){\to}H^2(\Tot(M^{*,*})).
\end{aligned}
\end{equation*}
For this computation  the relatively $F$-projective resolution used in Theorem \ref{FirstSecondPage} $(\mathbb{Z}\stackrel{\epsilon_Q}{\leftarrow}Q_i,\delta_i')$ will be as in Proposition \ref{prop:resultion relativa matched}, and the relatevely $G$-projective resolution $(\mathbb{Z}\stackrel{\epsilon_P}{\leftarrow}P_i,\delta_i)$ will be the total complex of the tensor product of the cyclic resolutions for $\langle t_1\rangle$ and $\langle t_2\rangle$, considered as a relatively $G$-projective $\Sigma$-resolution of $\Z$ using Lemma $\ref{Example2}$.

The $0^{th}$-page of the spectral sequence in Theorem \ref{FirstSecondPage} is given by
\[E_0^{i,j}=\Hom_{\Z\Sigma}(P_{i+1}\otimes Q_{j+1},\mathbb{C}^\times)=\Hom_{\Z\Sigma}\left(\bigoplus_{k=1}^{i+2}\Z F\otimes \Z G [G]^{j+1},\mathbb{C}^\times\right)\]

The horizontal and vertical differentials $d_h$ and $d_v$ are induced by the differentials of the resolutions $(\mathbb{Z}\stackrel{\epsilon_P}{\leftarrow}{P_i})$ and $(\mathbb{Z}\stackrel{\epsilon_Q}{\leftarrow}{Q_i})$ respectively.\\ 
Each $\Sigma$-module $\Z F\otimes \Z G[G]^{j+1}$ is free with basis $\{e\otimes[g_1|\cdots|g_{j+1}];e\neq g_1,\cdots g_{j+1}\in N\}$. 
Therefore, an element in $E_0^{i,j}$ is defined by a tuple $(h_1,\ldots,h_{i+2})$ with $h_k\in C^{j+1}(F,\mathbb{C}^\times)$, where $h_k(f_1,\cdots,f_{j+1})=1$ if any of the entries is the identity of $F$. 
Similarly, the differentials $d_0^{k,0}:\Hom_\Sigma(P^{k+1}\otimes Q^1,\mathbb{C}^\times)$ of the $0^{th}$-page, induced by the vertical differentials of the double complex are given by
\[d_0(f)(e\otimes[g_1|g_2])=f(e\otimes(g_1[g_2]-[g_1g_2]+[g_1])).\]
Since we are considering $\mathbb{C}^\times$ as a trivial $\Sigma$-module, the elements in $E_1^{k,0}=\Ker(d_0^{k,0})$ are in correspondence to tuples $(\chi_1,\ldots, \chi_{i+2})$ with $\chi_i \in \widehat{G}$.

First, we will compute $\Ker(d_2)$. By Lemma \ref{TwoGroups}, the group $E_2^{0,1}=\Der(F,H^2(G,\mathbb{C}^\times))$ is in correspondence with pairs $(A, B)$ of alternating $4\times 4$-matrices  satisfying the equations in \eqref{EquationsExample}.
According to Remark \ref{Rmk:matrices alternantes}(a), a representative element for $(A,B)$ in $ E^{0,1}_0=\Hom_\Sigma((\Z F\oplus \Z F)\otimes\Z G[G]^2,\mathbb{C}^\times)$  is defined by $\alpha =(\alpha_A,\alpha_B)$,
where $\alpha_{A}, \alpha_{B}$ are  $2$-cocycles defined in \eqref{2CocycleToAlternatingMatrix}.

By \eqref{Differentials}, we have  that $d_2(A,B)=\overline{d_h(\gamma)}$ were \[\gamma\in E^{1,0}_0= \Hom_\Sigma((\Z F\oplus \Z F\oplus \Z F)\otimes\Z G[G]),\] satisfies $d_h(\alpha_A,\alpha_B)=d_v(\gamma)$.  

By  \eqref{eq: diferencial 2 prod cyclic} we have that 
\begin{align}\label{eq:dAlphadBeta}
d_h(\alpha_A,\alpha_B)=(b_{M_1},b_{M_2},b_{M_3})\in E^{1,1}_0= \Hom_\Sigma((\Z F\oplus \Z F\oplus\Z F)\otimes\Z G[G]^2,\mathbb{C}^\times),    
\end{align}
where $b_{M_i}(x,y)=(-1)^{x^T M_i y}$ and

\begin{align*}
M_1=\widetilde{A}+F_1\widetilde{A}F_1^T=\left( \begin{array}{cccc}
0 & 0 & 0 & 0 \\
0 & 0 & 0 & 0\\
0 & 0 & b+d & d\\
0 & 0 & d & e\end{array} \right)\\\notag
M_2=\tilde{A}-F_2\tilde{A}F_2^T-(\tilde{B}-F_1\tilde{B}F_1^T)=\left( \begin{array}{cccc}
0 & 0 & 0 & 0 \\
0 & 0 & 0 & 0\\
0 & 0 & b+b'+d' & b+d+d'\\
0 & & b+d+d'   & c+e+e'\end{array}\right)\\\notag
M_3=\widetilde{B}+F_2\widetilde{B}F_2^T= \left( \begin{array}{cccc}
0 & 0 & 0 & 0 \\
0 & 0 & 0 & 0\\
0 & 0 &b' & c'\\
0 & 0 & c' & c'+e'\end{array} \right)
. \end{align*}
The cochain $\gamma=(\gamma_{M_1},\gamma_{M_2},\gamma_{M_3})\in E^{1,0}_0= \Hom_\Sigma((\Z F \oplus \Z F\oplus \Z F)\otimes\Z G[G]^2,\mathbb{C}^\times)$ defined by

\begin{align}\label{GammaABC}
    \gamma_{M_1}(\vect{x})&= \exp\Big (-\frac{\pi}{2}\left((b+d)x_3^2+2dx_3x_4+ex_4^2\right)\Big ) ,\notag\\
     \gamma_{M_2}(\vect{x})&= \exp\Big (-\frac{\pi}{2}\left((b+b'+d')x_3^2+2(b+d+d')x_3x_4+(c+e+e')x_4^2\right)\Big ),\notag\\
       \gamma_{M_3}(\vect{x})&=\exp\Big( -\frac{\pi}{2}\left( b'x_3^2+2c'x_3x_4+(c'+e')x_4^2\right)\Big),
\end{align}
satisfies \eqref{eq:dAlphadBeta}. Therefore,

\begin{align*}
\delta_h(\gamma_{M_1},\gamma_{M_2},\gamma_{M_3}) &= 
    \Big ( \frac{^{t_1}\gamma_{M_1}}{\gamma_{M_1}},\ \frac{^{t_2}\gamma_{M_1}\gamma_{M_2}(\ ^{t_1}\gamma_{M_2})}{\gamma_{M_1}},\frac{\gamma_{M_2}(\ ^{t_2}\gamma{M_2})\ ^{t_1}\gamma_{M_3}}{\gamma_{M_3}},\ \frac{^{t_2}\gamma_{M_3}}{\gamma_{M_3}}\Big )\\
    &=(1,\gamma_{M_2}^2,\gamma_{M_2}^2,1)\in \ker (d_1: E_1^{2,0} \to E_1^{3,0}).
\end{align*}

Since $G$ is an elementary abelian 2-group, we will use the canonical identification of $\widehat{G}$ with $G$. Under this identification we have that $\gamma_{M_2}^2=(0,0,b+b'+d',c+e+e')$. 

The pair $(A,B)$ belongs to $\Ker(d_2)$ if and only if $(0,\gamma_{M_2}^2,\gamma_{M_2}^2,0)$ belongs to the image of $d_1:E_1^{1,0}\to E^{2,0}_1$, if and only if there exists $(\mu_A,\mu_B,\mu_C)\in E_1^{1,0}=F^{\times 3}$ such that
$d_h(\mu_A,\mu_B,\mu_C)=(0,\gamma_{M_2}^2,\gamma_{M_2}^2,0)$.
This means that
\begin{align*}
(F_1-I)\mu_A&=0\\
(F_2-I)\mu_C&=0\\
(F_2-I)\mu_A+(F_1+I)\mu_{B}&=\gamma_{M_2}^2\\
(F_1-I)\mu_C+(F_2+I)\mu_{B}&=\gamma_{M_2}^2.
\end{align*}
From this equation we obtain
$b+b'+d'=c+e+e'=0$. Joining these two equations with  \eqref{EquationsExample} we get a system of equation with $5$ free variables, hence $\Ker(d_2)=(\Z/2)^5.$

Hence we have the exact sequence  
\begin{equation}\label{ShortSequence}
0\to H^2(F,G)\stackrel{}{\to}\Opext_\lhd(\mathbb{C}F,\mathbb{C}^G)\stackrel{\pi}{\to} \operatorname{Ker}(d_2)\to 0.\end{equation}


An element in $\Ker(d_2)$ is represented by a pair of matrices A,B as in \eqref{TwoGroups}. Let us assign $c'=1$ and consider the remaining variables zero, and let us call the respective pair of matrices $(A_c', B_c')$. A section of the sequence \eqref{ShortSequence} send $(A_c,B_c)$ to the class of the extension \[(\alpha_c,\gamma_c)\in H^1(\Hom_\Sigma(\Tot (D^{*,*},\mathbb{C}^*))\cong \Opext_\lhd(\mathbb{C}F,\mathbb{C}^G)\] where $\alpha_c$ is the $2$-cocycle associated to $(A_c',B_c')$ and $\gamma_c$ is given according to \ref{GammaABC}, by

\begin{align*}
    \gamma_{M_1}(\vect{x})&= \exp\Big (-\frac{\pi}{2}\left(x_3^2+x_4^2\right)\Big ) ,\notag\\
     \gamma_{M_2}(\vect{x})&= \exp\Big (-\frac{\pi}{2}\left((x_3^2\right)\Big ),\notag\\
       \gamma_{M_3}(\vect{x})&=\exp\Big( -\frac{\pi}{2}\left(x_4^2\right)\Big).
\end{align*}
It can be verified that the class of $(\alpha_c,\gamma_c)$ has order 4 in $\Opext_{\rhd,\lhd}(\mathbb{C}F,\mathbb{C}^G)$. In the same way, if we take the variable $d'$ to be $1$ and consider the remaining  variables null we get an element $(\alpha_d,\gamma_d)$ of order 4. Any other element outside the subgroup $\langle(\alpha_c,\gamma_c),(\alpha_d,\gamma_d)\rangle\cong (\Z/4)^2$ has order $2$, otherwise the order of $\Opext_{\rhd,\lhd}(\mathbb{C}F,\mathbb{C}^G)$ could not be $2^7$. That is why $\Opext_\lhd(\mathbb{C}F,\mathbb{C}^G)\cong (\Z/2)^3\oplus (\Z/4)^2$.
%

\end{proof}

\end{document}